\documentclass[12pt]{amsart}
\usepackage{amssymb}
\newtheorem{theorem}{Theorem}
\newtheorem{maintheorem}[theorem]{Main Theorem}
\newtheorem{corollary}[theorem]{Corollary}
\newtheorem{sublemma}{Lemma}[theorem]

\newcommand{\QED}{\end{proof}}

\def\BF#1.{{\bf #1.}}
\newcommand{\url}[1]{{\tt #1}}

\newcommand{\Vopenka}{Vop\v{e}nka}

\renewcommand{\P}{{\mathbb P}}

\newcommand{\R}{{\mathbb R}}

\newcommand{\of}{\subseteq}

\newcommand{\set}[1]{\{\,{#1}\,\}}

\newcommand{\Add}{\mathop{\rm Add}}

\newcommand{\satisfies}{\models}
\newcommand{\forces}{\Vdash}

\newcommand{\intersect}{\cap}

\newcommand{\smalllt}{\mathrel{\mathchoice{\raise2pt\hbox{$\scriptstyle<$}}{\raise1pt\hbox{$\scriptstyle<$}}{\raise0pt\hbox{$\scriptscriptstyle<$}}{\scriptscriptstyle<}}}
\newcommand{\smallleq}{\mathrel{\mathchoice{\raise2pt\hbox{$\scriptstyle\leq$}}{\raise1pt\hbox{$\scriptstyle\leq$}}{\raise1pt\hbox{$\scriptscriptstyle\leq$}}{\scriptscriptstyle\leq}}}
\newcommand{\lt}{\smalllt}
\newcommand{\ltomega}{{{\smalllt}\omega}}

\newcommand{\boolval}[1]{\mathopen{\lbrack\!\lbrack}\,#1\,\mathclose{\rbrack\!\rbrack}}
\def\[#1]{\boolval{#1}}

\newcommand{\UnderTilde}[1]{{\setbox1=\hbox{$#1$}\baselineskip=0pt\vtop{\hbox{$#1$}\hbox to\wd1{\hfil$\sim$\hfil}}}{}}
\newcommand{\Undertilde}[1]{{\setbox1=\hbox{$#1$}\baselineskip=0pt\vtop{\hbox{$#1$}\hbox to\wd1{\hfil$\scriptstyle\sim$\hfil}}}{}}
\newcommand{\undertilde}[1]{{\setbox1=\hbox{$#1$}\baselineskip=0pt\vtop{\hbox{$#1$}\hbox to\wd1{\hfil$\scriptscriptstyle\sim$\hfil}}}{}}
\newcommand{\UnderdTilde}[1]{{\setbox1=\hbox{$#1$}\baselineskip=0pt\vtop{\hbox{$#1$}\hbox to\wd1{\hfil$\approx$\hfil}}}{}}
\newcommand{\Underdtilde}[1]{{\setbox1=\hbox{$#1$}\baselineskip=0pt\vtop{\hbox{$#1$}\hbox to\wd1{\hfil\scriptsize$\approx$\hfil}}}{}}

\newcommand{\st}{\mid}

\def\<#1>{\langle#1\rangle}

\newcommand{\ZFC}{{\rm ZFC}}
\newcommand{\ZF}{{\rm ZF}}
\newcommand{\ZFA}{{\rm ZFA}}

\newcommand{\AC}{{\rm AC}}

\newcommand{\RR}{{\rm RR}}

\newcommand{\cell}[1]{\boxit{\hbox to 17pt{\strut\hfil$#1$\hfil}}}
\newcommand{\head}[2]{\lower2pt\vbox{\hbox{\strut\footnotesize\it\hskip3pt#2}\boxit{\cell#1}}}
\newcommand{\boxit}[1]{\setbox4=\hbox{\kern2pt#1\kern2pt}\hbox{\vrule\vbox{\hrule\kern2pt\box4\kern2pt\hrule}\vrule}}
\newcommand{\Col}[3]{\hbox{\vbox{\baselineskip=0pt\parskip=0pt\cell#1\cell#2\cell#3}}}
\newcommand{\tapenames}{\raise 5pt\vbox to .7in{\hbox to .8in{\it\hfill input: \strut}\vfill\hbox to
.8in{\it\hfill scratch: \strut}\vfill\hbox to .8in{\it\hfill output: \strut}}}
\newcommand{\Head}[4]{\lower2pt\vbox{\hbox to25pt{\strut\footnotesize\it\hfill#4\hfill}\boxit{\Col#1#2#3}}}
\newcommand{\Dots}{\raise 5pt\vbox to .7in{\hbox{\ $\cdots$\strut}\vfill\hbox{\ $\cdots$\strut}\vfill\hbox{\
$\cdots$\strut}}}
\newcommand{\df}{\it} 
\begin{document}
\author[Hamkins]{Joel David Hamkins}
\address{J. D. Hamkins, Mathematics,
The Graduate Center of The City University of New York, 365
Fifth Avenue, New York, NY 10016 \& Mathematics, The
College of Staten Island of CUNY, Staten Island, NY 10314;
and Department of Philosophy, New York University, 5
Washington Place, New York, NY 10003}
\email{jhamkins@gc.cuny.edu, http://jdh.hamkins.org}
\thanks{The research of the first author has been
supported in part by grants from the National Science
Foundation, the Simons Foundation and the CUNY Research
Foundation.}
\author[Palumbo]{Justin Palumbo}
\address{J. Palumbo, Department of Mathematics,
                    University of California at Los Angeles,
                    Los Angeles, California}
\email{justinpa@math.ucla.edu}
\keywords{set theory, forcing}
\begin{abstract}
The rigid relation principle, introduced in this article,
asserts that every set admits a rigid binary relation. This
follows from the axiom of choice, because well-orders are
rigid, but we prove that it is neither equivalent to the
axiom of choice nor provable in Zermelo-Fraenkel set theory
without the axiom of choice. Thus, it is a new weak choice
principle. Nevertheless, the restriction of the principle
to sets of reals (among other general instances) is
provable without the axiom of choice.
\end{abstract}

\title{The rigid relation principle, a new
weak choice principle} \maketitle

In this article we introduce the rigid relation principle
(\RR), which asserts that every set admits a rigid binary
relation, or equivalently, that every set is the vertex set
of a rigid directed graph. To be precise, \RR\ asserts that
for every set $A$ there is a binary relation $R$ on $A$
such that the structure $\<A,R>$ is {\df rigid}, meaning
that it has no nontrivial automorphisms, that is, no
bijective function $\pi:A\to A$ such that
$a\mathrel{R}b\iff \pi(a)\mathrel{R}\pi(b)$, other than the
identity function.

The \RR\ principle follows easily from the axiom of choice,
because under \AC\ every set $A$ has a well-order $\lt$,
and it is easy to see that well-orders are rigid: if a
bijection $\pi:A\to A$ is $\lt$-preserving, then there can
be no $\lt$-minimal element $a$ for which $\pi(a)\neq a$,
and so $\pi$ is the identity function. Nevertheless, we
shall prove that \RR\ is not equivalent to \AC, and neither
is it provable in \ZF, assuming \ZF\ is consistent. Thus,
it is a weak choice principle.

\begin{maintheorem}
 The rigid relation principle \RR\ is a weak choice principle in
 the sense that it is a consequence of the axiom of choice,
 but not equivalent to it, assuming \ZF\ is consistent, and
 neither is it provable in \ZF\ alone. Nevertheless,
 the restriction of \RR\ to sets of reals (among other
 general instances) is provable in \ZF.
\end{maintheorem}

In other words, the first sentence of the theorem asserts
that $\ZFC$ proves $\RR$, but $\ZF+\RR+\neg\AC$ and
$\ZF+\neg\RR$ are each consistent, if \ZF\ is. The theory
$\ZF+\RR$ therefore lies strictly between $\ZF$ and $\ZFC$,
if these theories are consistent. This research project
arose out of our inquiries and answers posted on
MathOverflow at
\cite{MO6262:DoesEverySetAdmitARigidBinaryRelation?}. The
claims made in the main theorem are proved separately in
theorems \ref{Theorem.SetsOfReals},
\ref{Theorem.RxBsubsetsAreRigid}, \ref{Theorem.ZF+notRR}
and \ref{Theorem.RR+notAC}, plus the observation above that
\AC\ easily implies \RR.

Before continuing, let us mention that a related classical
result of \Vopenka, Pultr and Hedrlin
\cite{VopenkaPultrHedrlin1965:RigidRelationEverySet}
establishes in \ZFC, using the axiom of choice, that every
set $A$ admits a binary relation $R$ with what we shall
call a {\df strongly rigid} relation, namely, the structure
$\<A,R>$ not only has no nontrivial automorphisms, but also
has no nontrivial endomorphisms, that is, no 
functions $f:A\to A$ with $a\mathrel{R} b \Rightarrow f(a)\mathrel{R}f(b)$,
other than the identity function. No infinite well-order is strongly rigid. It is
an interesting elementary exercise to see, without the
axiom of choice, that every countable set has a strongly
rigid relation.

\section{The \RR\ principle is provable in \ZF\ for sets of
reals}

Let us begin with the observation that the rigid relation
principle is outright provable in \ZF\ when it comes to
sets of reals.

\begin{theorem}[\ZF]\label{Theorem.SetsOfReals}
Every set of reals admits a rigid binary relation.
\end{theorem}

\begin{proof}
We shall not use the axiom of choice in this argument.
Suppose that $A$ is a set of reals. We may regard $A$ as a
subset of Cantor space $2^\omega$, since this space is
bijective with $\R$. We break into several cases.

If $A$ is countable, then we may impose an order structure
on it by making it a linear order isomorphic to $\omega$,
or a finite linear order if $A$ is finite, and these orders
are rigid.


Let us suppose next merely that $A$ has a countably
infinite subset. Enumerate this subset as $Z=\{\,
z^*,z_0,z_1,\ldots\,\}\of A$, where $z^*$ and the $z_n$ are
all distinct. For each finite binary sequence $s$, let
$U_s$ be the neighborhood in $2^\omega$ of all sequences
extending $s$, so that $U_s(x)\iff s\of x$. The structure
$\<A,U_s>_{s\in 2^\ltomega}$ is rigid, since if a point $x$
in $A$ is moved to another point, then it moves out of some
neighborhood $U_s$ that it was formerly in. We now reduce
this structure to a binary relation. Begin by enumerating
the finite binary sequences as $s_0$, $s_1$, and so on
(this does not require \AC). Let $R$ be the relation on $A$
that, first, places all the $z_n$ below $z^*$, ordered like
$\lt$ on $\omega$, and makes $R(z^*,z^*)$ true. That is, we
ensure that $R(z^*,z^*)$ holds and that $R(z_n,z^*)$ and
$R(z_n,z_{n+1})$ hold for every $n$. Next, for $y\notin Z$,
we define that $R(x,y)$ holds if and only if $x=z_n$ for
some $n$ and $U_{s_n}(y)$. That is, the first coordinate
gives you some $z_n$, and hence some $s_n$, and then you
use this to determine which neighborhood predicate to apply
to $y$, for $y$ outside of $Z$. Let us argue that the
structure $\<A,R>$ is rigid. First, observe that $z^*$ is
the only real for which $R(z^*,z^*)$, and that the only predecessors
of $z^*$ are of the form $z_n$. Each $z_n$ is thus
individually definable since $R$ orders them in order type
$\omega$. So $z^*$ and each $z_n$ are definable in $\<A,R>$
and therefore fixed by all automorphisms. Since every $z_n$
is fixed, it follows that every automorphism must respect
the neighborhood $U_{s_n}\intersect A$, and hence fix all
reals. So there are no nontrivial automorphisms, and $\<A,
R>$ is rigid, as desired.



The remaining case occurs when $A$ is uncountable, but has
no countably infinite subset. In other words, $A$ is
infinite and Dedekind finite. In this case, every
permutation of A will consist of disjoint orbits of finite
length, since if there were an infinite orbit, then we
could build a countably infinite subset of A by iterating
it. But if every permutation of A is like that, then A has
no permutations that respect a linear order, such as the
usual linear order $\lt$ of the reals. Thus, $\<A,\lt>$ is
rigid.
\end{proof}

It follows from this theorem that the easy observation that
\AC\ implies \RR\ because well-orders are rigid cannot be
reversed on a set-by-set basis.

\begin{corollary}\label{Corollary.RigidNotWO}
 It is relatively consistent with \ZF\ that there are sets
 that are not well-orderable, but which nevertheless have rigid binary
 relations.
\end{corollary}

\begin{proof}
The usual counterexamples to \AC\ in the symmetric forcing
models produce non-wellorderable sets of reals. (For
example, consider \cite[Lemma
5.15]{Jech1973:TheAxiomOfChoice}). But these sets admit
rigid binary relations by theorem
\ref{Theorem.SetsOfReals}.
\end{proof}

Let us now generalize theorem \ref{Theorem.SetsOfReals} to
the following result, of which we shall make use in the
proof of theorem \ref{Theorem.RR+notAC}, where we show that
\RR\ is not equivalent to \AC. We define a structure to be
{\df hereditarily rigid} if every substructure of it is
rigid. For example, every well-order is hereditarily rigid,
since any suborder of a well-order is still a well-order,
which is rigid. Also, every linearly ordered Dedekind
finite set is hereditarily rigid, since any subset of a
Dedekind finite set is Dedekind finite, and linearly
ordered such sets are rigid as we mentioned in the proof of
theorem \ref{Theorem.SetsOfReals}, since any bijection of a
Dedekind finite set has only finite orbits, and such a
bijection, if nontrivial, cannot respect a linear order. A
relation $R$ on $A$ is {\df irreflexive} if $R(a,a)$ never
holds for any $a\in A$.
%
%
%
%
%
%
%

\goodbreak
\begin{theorem}[\ZF]\label{Theorem.RxBsubsetsAreRigid}
 If a set $B$ has a hereditarily rigid irreflexive binary
 relation, then every subset
 $A\of \R\times B$ has a rigid binary relation. In
 particular, every subset $A\of\R\times\gamma$ for an
 ordinal $\gamma$ has a rigid binary relation.
\end{theorem}
%
%
%
%

\begin{proof} We shall not use the axiom of choice in
this argument. Suppose that $\<B,R_1>$ is hereditarily
rigid and irreflexive, and consider $A\of 2^\omega\times
B$. As in theorem \ref{Theorem.SetsOfReals}, we split into
cases. If $A$ is countable, then we may easily impose a
rigid linear order on $A$ as before.

For the main case, let us suppose again merely that $A$ has
a countably infinite subset $Z=\{\, z^*, z_0,
z_1,\ldots\,\}\of A$, where $z^*$ and all the $z_n$ are
distinct. Let $s_0$, $s_1$, and so on enumerate the finite
binary sequences. Following theorem
\ref{Theorem.SetsOfReals}, define the relation $R$ on $A$
so that $R(z^*,z^*)$ holds and so that $R(z_n,z^*)$ and
$R(z_n,z_{n+1})$ holds for every natural number $n$. For
$\<x,b>\in A-Z$, we define that $R(z_n,\<x,b>)$ holds if
and only if $s_n\subseteq x$, and for $\<x,b>,\<y,c>\in
A-Z$, we define that $R(\<x,b>,\<y,c>)$ holds if and only
if $R_1(b,c)$. We claim that $\<A,R>$ is rigid. To see
this, suppose that $\pi:A\rightarrow A$ is an
$R$-automorphism. Since $R_1$ is irreflexive, it follows
that $z^*$ is the only point for which $R(z^*,z^*)$ holds,
and so as in the proof of theorem
\ref{Theorem.SetsOfReals}, we see that $z^*$ and every
$z_n$ are definable in $\<A,R>$ and hence fixed by $\pi$.
From this, it follows that $\pi(\<x,b>)=\<y,c>$ implies
$x=y$, since this is true inside $Z$, and otherwise outside
$Z$, if $x\neq y$, then some $z_n$ will be $R$-related to
$\<x,b>$, but $\pi(z_n)=z_n$ is not related to $\<y,c>$.
Thus, $\pi$ fixes all the first coordinates, and therefore
on $A-Z$ amounts to an $R_1$-automorphism of the second
coordinates, on every slice of $A-Z$. Since every such
slice is a subset of $B$, it remains rigid under the
relation $R_1$. And so $\pi$ is the identity on $A$, as
desired.

Finally, for the remaining case, we suppose that $A$ has no
countably infinite subset. In this case, define
$R(\<x,b>,\<y,c>)$ if and only if $x<y$ in the reals, or
$x=y$ and $R_1(b,c)$. If $\pi:A\to A$ is an
$R$-automorphism, then since $A$ is Dedekind finite, we see
that $\pi$ consists of finite orbits. We claim that
$\pi(\<x,b>)=\<y,c>$ implies $x=y$. To see this, consider
first the case $x<y$, which implies that $\<x,b>$ is
$R$-related to $\pi(\<x,b>)$, and consequently by induction
each point in the orbit of $\<x,b>$ is $R$-related to the
next; but since the orbit is finite, these points
eventually return to $\<x,b>$ and therefore at some point
must have their first coordinates strictly descend,
preventing $R$ at that step of the orbit, a contradiction.
The case $y<x$ is similarly contradictory, and so $x=y$, as
desired. Thus, $\pi$ amounts to an automorphism on the
second coordinates of each slice of $A$, and since $R_1$ is
hereditarily rigid, it follows that $\pi$ is the
identity.\end{proof}

\section{The \RR\ principle is not provable in \ZF}

We prove next that the rigid binary relation principle is
not provable in \ZF, assuming \ZF\ is consistent, and
therefore it may be viewed as a nontrivial choice
principle.

\begin{theorem}\label{Theorem.ZF+notRR}
If \ZF\ is consistent, then so is $\ZF+\neg\RR$.
Specifically, every model of \ZF\ admits a symmetric
extension in which \RR\ fails.
\end{theorem}

\begin{proof}
We shall make use of the permutation model technique for
constructing models of $\ZF$ set theory with violations of
\AC. It will be enough to show that any model of \ZF\ can be
extended to a permutation model $M\satisfies\ZFA$, that is,
Zermelo-Fraenkel set theory with atoms, in which the set of atoms
$A$ has no rigid binary relation. To see that this suffices, suppose we have
such a model $M$. Observe first that the statement that
every binary relation on $A$ has a nontrivial automorphism
can be expressed using the $\in$ relation and quantifiers
bounded by $\mathcal{P}^5(A)$, the fifth iteration of the
powerset operation on $A$. Meanwhile, the Jech-Sochor
theorem \cite[Theorem 6.1]{Jech1973:TheAxiomOfChoice}
provides a symmetric extension $N\satisfies\ZF$ containing a
set $B$ for which $\langle\mathcal{P}^5(A),\in\rangle^M$ is
isomorphic to $\langle\mathcal{P}^5(B),\in\rangle^N$. It
follows therefore that \RR\ fails in $N$, and we will have
thus achieved our desired model. 

The model $M$ we shall present will be the basic Fraenkel
model described in \cite[section
4.3]{Jech1973:TheAxiomOfChoice}, which is determined inside
a fixed \ZFA\ model with a countably infinite set $A$ of atoms by
the group $\mathcal{G}$ of all permutations of $A$ and the
filter $\mathcal{F}$ generated by the subgroups of the form
$\textrm{fix}(E)=\{\pi\in\mathcal{G}:\pi(a)=a\textrm{ for
all }a\in E\}$, where $E$ is any finite subset of $A$.
Note that any automorphism $\pi$ of $A$ extends hereditarily in
the natural way from $A$ to the universe built over $A$ by
the $\in$-recursive definition $\pi(x)=\set{\pi(y)\st y\in
x}$, and so $\pi(x)$ makes sense for arbitrary sets in our
\ZFA\ universe with atoms $A$. We define that such a set
$x$ belongs to $M$ exactly when there is some finite $E\of
A$ for which $\textrm{fix}(E)\subseteq\textrm{sym}(x)$,
where $\textrm{sym}(x)=\{\pi\in\mathcal{G}:\pi(x)=x\}$.
Note that when $R$ is a binary relation on $A$, the set
$\textrm{sym}(R)$ consists precisely of the automorphisms
of $\langle A,R\rangle$. 

The resulting model $M$ satisfies \ZFA\ for reasons
explained in \cite{Jech1973:TheAxiomOfChoice}. This model $M$
can be realized as an extension of any model of $\ZF$
by first adjoining countably many atoms and then taking
the submodel obtained by using $\mathcal{G}$ with $\mathcal{F}$
as described. We claim now additionally that in $M$ there is
no rigid binary relation on $A$. To see this, suppose
$R\subseteq A\times A$ is in $M$. Thus, there is a finite set $E\subseteq A$ such that
whenever $\pi\in\textrm{fix}(E)$, then
$\pi\in\textrm{sym}(R)$. In other words, any automorphism
$\pi$ of $A$ fixing the finite set $E$ pointwise will be an
automorphism of $\<A,R>$. Since $A$ is infinite, we may
easily find an automorphism $\pi$ of $A$ that swaps two
atoms outside $E$. Since $\pi$ fixes $E$, it follows that
$\pi$ is an automorphism of $\<A,R>$; and since $\pi$ has
finite support, it follows that $\pi\in M$. Thus, $M$ has a
nontrivial automorphism of $\langle A,R\rangle$, and
consequently \RR\ must fail in $M$, as desired.
\end{proof}

It follows from the proof of theorem \ref{Theorem.ZF+notRR}
and by examining the proof of the Jech-Sochor theorem that
our theorem \ref{Theorem.SetsOfReals} is optimal in the
sense that while it asserts that every set of reals admits
a rigid binary relation, this cannot be improved in general
to sets of sets of reals. This is because models  of \ZFA\
are simulated in the Jech-Sochor theorem by \ZF\ models by
a map transferring the set of atoms to a set of sets of
reals. Specifically, the set $B$ in the proof of theorem
\ref{Theorem.SetsOfReals} for which
$\langle\mathcal{P}^5(A),\in\rangle^M$ is isomorphic to
$\langle\mathcal{P}^5(B),\in\rangle^N$ is a member of
$\mathcal{P}^2(\mathbb{R})$. So although theorem
\ref{Theorem.SetsOfReals} is proved in \ZF\ for sets of
reals, the conclusion need not hold for sets of sets of
reals. (This argument also shows that the Jech-Sochor
theorem cannot be improved to map the atoms into the reals,
a fact which has already been known by other means.)

We call attention to the curious fact we noticed that all
of our provable instances of the rigid relation principle
have a hereditary nature; that is, for the sets we have
shown in \ZF\ to admit rigid binary relations in theorems
\ref{Theorem.SetsOfReals} and
\ref{Theorem.RxBsubsetsAreRigid}, their subsets also admit
rigid binary relations. Could this be due to a general
phenomenon, for which every set with a rigid binary
relation has every subset also having a rigid binary
relation? One wouldn't think so, but we have yet no
counterexample. The general phenomenon holds trivially, of
course, in models of \RR, but let us point out that it also
holds in the basic Fraenkel model $M$ of ZFA used in the
proof of theorem \ref{Theorem.ZF+notRR}, where \RR\ fails.
This is because in $M$, we claim, the sets admitting rigid
binary relations are precisely the wellorderable sets. This
follows from a result of Blass
\cite{Blass1977:RamseysTheoremintheHierarchyofChoicePrinciples},
showing that any non-wellorderable set in $M$ has an
infinite subset that is equinumerous with a subset of the
atoms $A$. In this case, the argument from theorem
\ref{Theorem.ZF+notRR} showing that $A$ has no rigid binary
relation in $M$ then shows that any non-wellorderable set
in $M$ admits no rigid binary relation.

\section{\RR\ is strictly weaker than \AC}

So far, we have observed that \RR\ is provable from the
axiom of choice, but not without. In this section, we prove
nevertheless that \RR\ is not equivalent to \AC. So it is a
strictly weaker choice principle.

\begin{theorem}\label{Theorem.RR+notAC}
 If\/ \ZF\ is consistent, then so is $\ZF+\RR+\neg\AC$.
 Indeed, \RR\ holds in the Cohen model of $\ZF+\neg\AC$.
\end{theorem}

\begin{proof}
We shall apply theorem \ref{Theorem.RxBsubsetsAreRigid} to
show that \RR\ holds in the symmetric Cohen model $M$ of
$\ZF+\neg\AC$. Symmetric models are similar to permutation
models, but built instead with a group $\mathcal{G}$ of
automorphisms of a notion of forcing $\mathbb{P}$. For the
Cohen model, let $\mathbb{P}=\Add(\omega,\omega)$ be the
usual forcing notion to add $\omega$ many Cohen reals, that
is, the partial order consisting of the finite partial
functions from $\omega\times\omega$ into $\{0,1\}$, ordered
by reverse inclusion. Every permutation
$\pi:\omega\rightarrow\omega$ induces an automorphism
$\pi:\mathbb{P}\rightarrow\mathbb{P}$, defined by permuting
the Cohen reals in the same way, so that
$$\pi(p)(\pi(m),n)=p(m,n).$$  
Let $\mathcal{G}$ be the collection of these induced
automorphisms of $\P$. Note that any automorphism of $\P$
induces hereditarily a corresponding automorphism of the
class of $\P$-names $\tau\mapsto\tau^\pi$, defined
recursively by $\tau^\pi=\set{\<\sigma^\pi,\pi(p)>\st
\<\sigma,p>\in\tau}$, and one may prove inductively that
$p\forces_\P\varphi(\tau)$ if and only if
$\pi(p)\forces_\P\varphi(\tau^\pi)$. Let $M$ be the
submodel of the generic extension $V[G]$ of elements
$\tau_G$ having names $\tau\in V^\P$ which are symmetric in
the sense that there is some finite $E\subseteq\omega$ such
that whenever $\pi$ fixes each member of $E$, then
$\tau=\tau^\pi$. Jech \cite[Lemma
5.25]{Jech1973:TheAxiomOfChoice} establishes the following
property of this model, which we state without proof.

\begin{sublemma}
In the Cohen model $M$, there is a set of reals $A$ such
that every set can be injected into
$A^{<\omega}\times\gamma$ for some ordinal $\gamma$.
\end{sublemma}

\noindent It follows from this that every set in $M$ can be
injected into $\R^\ltomega\times\gamma$ and indeed into
$\R\times\gamma$ for some ordinal $\gamma$. Thus, every set
in $M$ is bijective with a subset of $\R\times\gamma$. But
theorem \ref{Theorem.RxBsubsetsAreRigid} establishes that
all such sets admit a rigid binary relation, and so \RR\
holds in $M$, as desired.\end{proof}

This completes the proof of theorem \ref{Theorem.RR+notAC}
and therefore also of the main theorem stated at the
beginning of the paper.

Before concluding the paper, let us make a few small
remarks about how \RR\ relates to some of the other
well-known weak choice principles. Since we have proved
that \RR\ holds in the Cohen model, \RR\
fails to imply any of the choice principles that fail in
that model. For example, \RR\ is consistent with the
existence of infinite Dedekind finite sets of reals, since
such sets exist in that model. In particular, it follows
that \RR\ does not imply the axiom of countable choice or
the axiom of dependent choices, as these fail in the Cohen
model. Meanwhile, \RR\ does not follow from $\AC_\omega$,
or from $\AC_\kappa$ for any cardinal $\kappa$, because
\RR\ fails but these axioms hold 
in the Fraenkel model arising from countable support
(or size $\kappa$ support) in place of finite support in
the proof of theorem \ref{Theorem.ZF+notRR}. We note that
in all of the models of \RR\ provided in this article,
every set can be linearly ordered, and it would be
interesting to separate this principle from \RR. Also, we
have not yet separated \RR\ from the prime ideal theorem.

%
%
%
%

We close the paper by mentioning a few of the many
directions in which one could hope to continue this
research. It would be natural to consider the rigid
relation principle in higher arities, so that one gets
rigid ternary relations or rigid $k$-ary relations on every
set. Similarly, one could consider the principles asserting
that every set has a rigid first-order structure in a
finite language or in a countable language. Returning to
binary relations, it would be natural to consider the
question of whether every set has a rigid linear order, or
a rigid symmetric relation. Another way to describe this
last idea is that although the rigid relation principle
asserts that every set is the vertex set of a rigid
directed graph, it is also natural to insist that every set
is the vertex set of a rigid graph. In addition, with each
of these variant rigidity principles one might also
consider the notion of hereditary rigidity that we
introduced above, as well as strong rigidity (no
endomorphisms rather than merely no automorphisms).

\bibliographystyle{alpha}
\bibliography{RigidRelation}

\end{document}